\documentclass[11pt,letterpaper]{article}

\usepackage[letterpaper,left=1.25in,right=1.25in,top=1.2in,bottom=1.2in,headsep=0.3in]{geometry}
\usepackage{iftex}
\usepackage{fontspec}
\usepackage{amsmath,amssymb,amsthm}
\usepackage{unicode-math}
\usepackage[final,protrusion=true,expansion=false]{microtype}
\usepackage[english]{babel}
\usepackage[autostyle=true]{csquotes}
\usepackage{xcolor}
\definecolor{navy}{RGB}{31,56,100}
\definecolor{rule}{RGB}{120,120,120}
\definecolor{fignavy}{RGB}{31,56,100}
\definecolor{figred}{RGB}{176,58,46}
\usepackage{booktabs}
\usepackage{caption}
\usepackage{enumitem}
\setlist[enumerate]{leftmargin=*,itemsep=0.25em}
\setlist[itemize]{leftmargin=*,itemsep=0.25em}
\usepackage{graphicx}
\usepackage{float}
\usepackage[section]{placeins}
\usepackage{adjustbox}
\usepackage{tikz}
\usetikzlibrary{arrows.meta,calc,decorations.pathreplacing,positioning,angles,quotes,decorations.markings}
\usepackage{listings}
\usepackage{needspace}
\usepackage[most]{tcolorbox}
\usepackage{fancyhdr}
\usepackage{titlesec}
\usepackage{setspace}

\titleformat{\section}{\color{navy}\large\bfseries}{\thesection.}{0.6em}{}
\titleformat{\subsection}{\color{navy}\normalsize\bfseries}{\thesubsection.}{0.6em}{}
\titleformat{\paragraph}[runin]{\normalfont\scshape}{}{0em}{}[]
\titlespacing*{\section}{0pt}{1.6\baselineskip plus 0.4\baselineskip minus 0.2\baselineskip}{0.7\baselineskip}
\titlespacing*{\subsection}{0pt}{1.2\baselineskip plus 0.3\baselineskip}{0.5\baselineskip}
\titlespacing*{\paragraph}{0pt}{0.8\baselineskip}{0.6em}

\newtheoremstyle{scplain}{0.9\baselineskip plus 0.2\baselineskip}{0.9\baselineskip plus 0.2\baselineskip}{\itshape}{}{\scshape}{.}{0.6em}{\thmname{#1}\thmnumber{ #2}\thmnote{ (\normalfont\itshape#3)}}
\newtheoremstyle{scremark}{0.9\baselineskip plus 0.2\baselineskip}{0.9\baselineskip plus 0.2\baselineskip}{\normalfont}{}{\scshape}{.}{0.6em}{\thmname{#1}\thmnumber{ #2}\thmnote{ (\normalfont\itshape#3)}}
\theoremstyle{scplain}
\usepackage{aliascnt}
\newtheorem{theorem}{Theorem}
\newaliascnt{lemma}{theorem}\newtheorem{lemma}[lemma]{Lemma}\aliascntresetthe{lemma}
\newaliascnt{proposition}{theorem}\newtheorem{proposition}[proposition]{Proposition}\aliascntresetthe{proposition}
\newtheorem{corollary}{Corollary}[theorem]
\theoremstyle{scremark}
\newaliascnt{remark}{theorem}\newtheorem{remark}[remark]{Remark}\aliascntresetthe{remark}
\makeatletter
\renewenvironment{proof}[1][\proofname]{\par\pushQED{\qed}\normalfont\topsep6\p@\@plus6\p@\relax\trivlist\item[\hskip\labelsep\scshape#1\@addpunct{.}]\ignorespaces}{\popQED\endtrivlist\@endpefalse}
\makeatother

\newtcolorbox{practicebox}[1]{enhanced,breakable,colback=white,colframe=navy,boxrule=0.5pt,arc=0pt,
  left=8pt,right=8pt,top=6pt,bottom=6pt,fonttitle=\scshape,coltitle=navy,colbacktitle=white,
  title={#1},attach boxed title to top left={yshift=-2.5mm,xshift=6mm},
  boxed title style={boxrule=0pt,colback=white,arc=0pt,left=3pt,right=3pt}}

\usepackage[style=authoryear-comp,backend=biber,maxcitenames=4,maxbibnames=10,giveninits=true,
  uniquename=false,uniquelist=false,useprefix=true,doi=true,isbn=false,url=false,eprint=false,dashed=false,natbib=false]{biblatex}
\DeclareFieldFormat{doi}{\href{https://doi.org/#1}{doi:\nolinkurl{#1}}}
\AtEveryBibitem{\clearfield{note}}
\renewbibmacro{in:}{}
\DeclareNameAlias{sortname}{family-given}
\DeclareFieldFormat[article]{volume}{#1}
\DeclareFieldFormat[article]{number}{\mkbibparens{#1}}
\renewbibmacro*{volume+number+eid}{\printfield{volume}\printfield{number}\setunit{\addcomma\space}\printfield{eid}}
\DeclareFieldFormat[article]{pages}{#1}
\DeclareFieldFormat[book]{series}{#1}

\usepackage[colorlinks=true,linkcolor=navy,citecolor=navy,urlcolor=navy,bookmarksnumbered=true,
  pdfstartview=FitH,breaklinks=true]{hyperref}
\hypersetup{
  pdftitle={Signal Correlation, IC, and PnL Dependence},
  pdfauthor={Marc Nunes},
  pdfsubject={A partial-correlation decomposition with residual-gauge invariants and ensemble bounds under Euclidean and risk metrics},
  pdfkeywords={information coefficient; signal correlation; PnL correlation; partial correlation; residual gauge; Gram matrix; multiple correlation; alpha combination; risk metric; quantitative equity},
  pdfcreator={LuaLaTeX}}
\usepackage[capitalise,nameinlink,noabbrev]{cleveref}
\crefname{lemma}{Lemma}{Lemmas}
\crefname{theorem}{Theorem}{Theorems}
\crefname{proposition}{Proposition}{Propositions}
\crefname{corollary}{Corollary}{Corollaries}
\crefname{remark}{Remark}{Remarks}
\crefname{section}{Section}{Sections}
\crefname{appendix}{Appendix}{Appendices}
\crefname{figure}{Figure}{Figures}
\crefname{table}{Table}{Tables}

\newcommand{\papertitle}{Signal Correlation, IC, and PnL Dependence}
\newcommand{\papersubtitle}{A partial-correlation decomposition with residual-gauge invariants and ensemble bounds under Euclidean and risk metrics}
\newcommand{\paperabstract}{Signal correlation and PnL correlation are correlations over different index sets—across assets at each date versus across dates for scalar payoffs—and practitioners often treat the first as a proxy for the second. We give an exact decomposition that shows what that proxy sees and what it discards. We recall that at each date a normalized signal's projection onto the realized demeaned return direction is its realized cross-sectional Pearson IC, so that fixed-signal PnL is return dispersion times IC \parencite{qian2004active}, and we identify the normalized similarity of two signals in the orthogonal complement as their partial correlation controlling for realized returns; the residual rotational freedom is an orthogonal gauge whose invariants are the transverse Gram matrix. Signal correlation therefore equals an uncentered IC cross-moment plus transverse similarity, while Pearson PnL correlation centers and dispersion-weights the IC series alone. Our main result is a non-identifiability theorem: absent constraints on transverse geometry, neither correlation bounds or orders the other, which sharpens the simulation finding of \textcite{sorensen2004multiple} into an exact statement. For normalized ensembles the transverse Gram matrix re-enters through the normalization denominator; the ex post optimal combination under Euclidean and general covariance-risk metrics is the classical multiple-correlation bound, and we show that the two rules agree for every longitudinal exposure if and only if the risk metric is isotropic on the signal span. The classical exact null law of partial correlation under isotropy, a synthetic mechanism illustration at weak-IC scales, and inference guidance under temporal dependence complete the treatment.}

\fancypagestyle{plain}{\fancyhf{}\fancyfoot[C]{\small\thepage}}

\begin{document}
\thispagestyle{plain}
\begin{center}
  {\color{navy}\LARGE\bfseries \papertitle\par}
  \vspace{0.6em}
  {\large\itshape \papersubtitle\par}
  \vspace{1.4em}
  {\normalsize Marc Nunes\par}
  \vspace{0.2em}
  {\small AlphaNova\quad\textperiodcentered\quad \href{mailto:marc.nunes@alphanova.tech}{marc.nunes@alphanova.tech}\par}
  \vspace{0.8em}
  {\small\scshape Research Note\quad\textperiodcentered\quad Version 8\quad\textperiodcentered\quad September 2026\par}
\end{center}
\vspace{1.2em}
\begin{center}
\begin{minipage}{0.9\textwidth}
\small
{\scshape Abstract.}\; \paperabstract
\end{minipage}
\end{center}
\vspace{1.0em}

\Needspace*{8\baselineskip}
\section{Setup: two correlations over different index sets}\label{sec:1}

Take $d\ge2$ assets observed on dates $t=1,\ldots,T$. Write $R_t\in\mathbb R^d$ for forward realized returns and $S_{i,t}\in\mathbb R^d$ for a cross-sectional signal. To begin, assume

\[
\mathbf1^\top S_{i,t}=0,
\qquad
\lVert S_{i,t}\rVert_2=1.
\]

Define the neutral subspace

\[
H=\{x\in\mathbb R^d:\mathbf1^\top x=0\},
\qquad \dim H=d-1,
\]

and split returns as

\[
R_t=\bar R_t\mathbf1+\widetilde R_t.
\]

Whenever $\widetilde R_t\ne0$, set

\[
a_t=\lVert\widetilde R_t\rVert_2,
\qquad
Q_t=\frac{\widetilde R_t}{a_t}\in S^{d-2}.
\]

Here $S^{d-2}:=\{x\in H:\lVert x\rVert_2=1\}$ is the unit sphere of $H$; since $\dim H=d-1$, it is a $(d-2)$-sphere in the usual convention that $S^n$ sits in an $(n+1)$-dimensional space. Thus $a_t$ is the realized cross-sectional return magnitude and $Q_t$ its direction. Dates with $a_t=0$ give zero cross-sectional PnL for every neutral signal and can be dropped.

Signal similarity is computed across assets at each date and then averaged over time. PnL correlation is computed across time, after each signal vector has been collapsed to a scalar payoff. These are not two estimators of one population correlation. They are inner products over different index sets.

\subsection{Motivation and related work}\label{sec:1.1}

Quantitative equity research runs on two correlation matrices that are often used interchangeably. The first is the \emph{signal} correlation matrix: for a library of candidate alphas, the average cross-sectional correlation between each pair of score vectors. It is the standard tool for judging whether a new signal is \enquote{different enough} to be worth adding, and for diagnosing crowding within a feature pipeline. The second is the \emph{PnL} correlation matrix: the time-series correlation of the realized payoffs that each signal generates. It is what an allocator, a risk manager, or a multi-strategy platform actually uses to size positions and to decide whether two books diversify each other. The literature does distinguish the two. \textcite{dang2019alphacorr} present alpha PnL correlation and alpha value correlation as alternative techniques for the same task, with the explicit requirement that a good correlation measure be able to predict the co-movement of two alphas' PnL. What is not stated is the condition under which one may stand in for the other. Practitioners know from experience that the two matrices disagree—low signal correlation does not guarantee low PnL correlation, and high signal correlation is sometimes harmless—but the disagreement is usually treated as noise, or as an artefact of sample size, rather than as something with an exact structure.

The framework that connects them is old. \textcite{grinold1989fundamental,grinold1994alpha} and \textcite{grinold2000active} express expected active return as a product of volatility, information coefficient, and score; the information coefficient is a cross-sectional correlation between forecast and realization, and \cref{lem:1} below is simply the realized, single-date version of that statement for a Euclidean-normalized portfolio. \textcite{clarke2002constraints} introduce the transfer coefficient to account for the gap between the research signal and the implemented portfolio; the scope condition following \cref{lem:1} is the point at which that gap enters the present decomposition. \textcite{qian2004active} already derive the single-period relationship between excess return, IC, and cross-sectional return dispersion, and analyze the interaction between IC and dispersion over time; \cref{lem:1} and the dispersion-regime term of \cref{sec:5} are the realized, Euclidean-normalized forms of those two observations. \textcite{qian2007quantitative} treat IC as a time series whose variability, not only its mean, drives active risk, and combine factors through their IC covariance; \cref{sec:5} and \cref{rem:5} make the same time-series distinction in the present notation. \textcite{ding2017redux} give a modern treatment of random, time-varying IC and the fundamental law. The alpha-combination literature \parencite{kakushadze2017billion} treats optimal combination weights through an inverse alpha-return covariance; \cref{sec:6} gives the single-date cross-sectional analogue in both the Euclidean and the risk metric. Distribution theory for the sample cross-sectional information coefficient, including its time-varying and risk-adjusted forms, is developed in \textcite{ding2023timevarying}; weighting schemes for combining correlated signals are treated in \textcite{grinold2010signal}. \textcite{sorensen2004multiple} study exactly the two quantities compared here, varying the cross-sectional correlation of alpha signals and the time-series correlation of their information coefficients in simulation, and report that the latter matters more for diversification of strategy risk. That finding is the empirical phenomenon this paper explains: the decomposition below identifies which component each measure sees, and therefore why the two behave differently. The present paper begins from the Qian–Hua economic identity and decomposes the relationship between signal correlation and PnL dependence into longitudinal IC co-alignment, transverse partial correlation, dispersion weighting, and temporal centering, in an invariant and multi-signal formulation.

The statistical objects are equally classical. The normalized transverse statistic of \cref{sec:3} is a cross-sectional partial correlation with a single control, whose exact null distribution goes back to \textcite{fisher1924partial} and is treated in \textcite{anderson2003multivariate}; \cref{sec:8} does not derive a new law but identifies which classical one applies and with which degrees of freedom. Partial correlation is also established in finance as a tool for conditioning asset relationships on other assets \parencite{kenett2015partial}; the use here is different, conditioning two signals on the realized return vector rather than relating assets to one another. The classification of configurations up to orthogonal transformation by their Gram matrix is closely related to Weyl's first fundamental theorem for the orthogonal group \parencite{weyl1939classical}, which states that polynomial orthogonal invariants are generated by inner products; \cref{prop:4} gives the elementary orbit statement directly in the transverse subspace, so that every frame-independent diagnostic factors through $(p,B)$. The isotropic reference distribution on the sphere in \cref{sec:8} is standard directional statistics \parencite{mardia2000directional}. Dependence-aware inference for time-averaged moments—HAC covariance estimation \parencite{newey1987simple} and block or stationary bootstraps \parencite{kunsch1989jackknife,politis1994stationary,lahiri2003resampling}—is the appropriate machinery for \cref{sec:11}, given signal persistence and overlapping forward horizons. Multiplicity across the $K(K-1)/2$ pairs of a library sits within the broader literature on data snooping and multiple testing in alpha discovery \parencite{white2000reality,romano2005stepwise,harvey2016cross,bailey2014deflated}, and \textcite{benjamini1995fdr} supplies the false-discovery-rate control named in \cref{sec:11}.

What the present paper adds is not a new identity but a bookkeeping discipline that these strands have not been combined into. Signal correlation, realized IC, partial correlation given returns, transfer efficiency, PnL correlation, multiple correlation, and risk-adjusted combination weights are shown to be coordinates and transformations of one date-by-date orthogonal decomposition. That decomposition says exactly which component each diagnostic sees, which it discards, and where the discarded component re-enters when signals are combined into a normalized or risk-constrained book. The decomposition also makes it impossible to confuse the Euclidean geometry in which Pearson IC lives with the covariance geometry in which economic risk lives; the two rules agree for every longitudinal exposure only in the isotropic case identified in \cref{sec:6}, and the paper keeps them as separate lenses.

Concretely, the paper contributes four things. First, \cref{thm:9}, an exact non-identifiability result: without constraints on transverse geometry, neither signal correlation nor PnL correlation bounds or orders the other. This sharpens the simulation finding of \textcite{sorensen2004multiple} into a statement about what is, and is not, determined. Second, the four-line hierarchy of \cref{sec:12} as a single organizing identity, together with the three-way attribution of \cref{sec:5}, which separates dispersion weighting, Pearson centering, and time-series normalization and isolates the pairwise dispersion-regime term. Third, the observation that the transverse Gram matrix, inert for a fixed signal, re-enters through the normalization denominator once signals are combined. Fourth, the characterization that the Euclidean and risk optima agree for every longitudinal exposure if and only if the risk metric is isotropic on the signal span.

\Needspace*{5\baselineskip}
\begin{lemma}[the longitudinal coordinate is realized Pearson IC]\label{lem:1}
For neutral, Euclidean-unit signals,

\[
\rho_{12,t}^{\mathrm{sig}}
:=\operatorname{Corr}_{\mathrm{cs}}(S_{1,t},S_{2,t})
=\langle S_{1,t},S_{2,t}\rangle,
\]

and

\[
\boxed{
p_{i,t}:=\langle S_{i,t},Q_t\rangle
=\operatorname{Corr}_{\mathrm{cs}}(S_{i,t},R_t)
=:\rho_{iR,t}.}
\]

So $p_{i,t}$ is exactly the realized cross-sectional Pearson IC. No $d$-versus-$(d-1)$ caveat is needed: the covariance and variance convention cancels in a Pearson correlation.

Trading the signal vector itself gives

\[
\boxed{
\Pi_{i,t}=\langle S_{i,t},R_t\rangle=a_tp_{i,t}.}
\]

\end{lemma}

Realized PnL is realized cross-sectional return magnitude times realized Pearson IC. This is the realized, Euclidean-normalized form of the single-period identity in \textcite{qian2004active}.

\paragraph{Important scope condition.}
The identity needs the same Euclidean-normalized vector to serve as both the traded portfolio and the argument of the IC. When the implemented portfolio $w_t$ departs from the research signal—through constraints, an optimizer, cost penalties, or a different normalization—set $\widehat w_t=w_t/\lVert w_t\rVert_2$. Then

\[
\Pi_t
=a_t\lVert w_t\rVert_2
\operatorname{Corr}_{\mathrm{cs}}(\widehat w_t,R_t).
\]

The alignment $\langle\widehat w_t,S_t\rangle$ between implemented portfolio and signal is close in spirit to the transfer coefficient of \textcite{clarke2002constraints}. Under unit-gross or unit-ex-ante-volatility scaling, $\lVert w_t\rVert_2$ varies from date to date and becomes a third multiplicative channel.

Rank IC does not satisfy $\Pi=a\times\mathrm{rankIC}$ unless the traded weights are themselves the normalized rank-transformed vector. A diagnostic should never quietly swap one signal representation for another.

\Needspace*{8\baselineskip}
\section{Moving frames and an elementary projection}\label{sec:2}

Write $O(H)$ for the orthogonal group of $H$: the linear maps $G\colon H\to H$ with $\langle Gx,Gy\rangle=\langle x,y\rangle$ for all $x,y\in H$. Since $\dim H=d-1$, $O(H)\cong O(d-1)$. Fix once and for all a reference unit vector $e_1\in S^{d-2}$. Any choice will do, for instance $e_1=(1,-1,0,\ldots,0)/\sqrt2$; the ambient basis vector $(1,0,\ldots,0)$ is not admissible because it does not lie in $H$. Nothing below depends on the choice.

A \emph{moving frame} is a choice, at each date, of an orthonormal basis of $H$ whose first element is the realized return direction $Q_t$. Equivalently, it is a choice of $G_t\in O(H)$ with

\[
G_tQ_t=e_1,
\]

the rotation (or reflection) that carries the date-$t$ frame onto the fixed reference basis. Such a $G_t$ exists for every $Q_t$ because $O(H)$ acts transitively on $S^{d-2}$, and it is unique only up to the stabilizer of $e_1$ in $O(H)$, namely $O(e_1^\perp)\cong O(d-2)$, where $e_1^\perp$ denotes the orthogonal complement of $e_1$ within $H$. That leftover freedom is what we call the \emph{residual gauge}; the word is used only in this weak sense of a redundant frame choice with a residual symmetry group, and no connection or curvature is introduced. In the frame, write

\[
P_{i,t}=G_tS_{i,t}=p_{i,t}e_1+U_{i,t},
\qquad U_{i,t}\in e_1^\perp.
\]

None of the diagnostics needs $G_t$ itself. They need only

\[
\langle S_{1,t},S_{2,t}\rangle,
\qquad
\langle S_{i,t},Q_t\rangle,
\qquad
a_t.
\]

\Needspace*{5\baselineskip}
\begin{lemma}[longitudinal–transverse decomposition]\label{lem:2}
For every date,

\[
\boxed{
\langle S_{1,t},S_{2,t}\rangle
=p_{1,t}p_{2,t}+\langle U_{1,t},U_{2,t}\rangle,}
\]

with

\[
\lVert U_{i,t}\rVert_2^2=1-p_{i,t}^2.
\]

\end{lemma}

This is just the orthogonal projection identity. Its value here is diagnostic, not algebraic.

Define time averaging once and for all by

\[
\overline{x}:=\frac1T\sum_{t=1}^T x_t.
\]

{\predisplaypenalty=10000 \postdisplaypenalty=10000 \interlinepenalty=10000 \relax
Then
\[
\boxed{
C_{\mathrm{sig}}=C_{\mathrm{IC}}+C_\perp,}
\]
where
\[
C_{\mathrm{sig}}=\overline{\langle S_{1,t},S_{2,t}\rangle},
\quad
C_{\mathrm{IC}}=\overline{p_{1,t}p_{2,t}},
\quad
C_\perp=\overline{\langle U_{1,t},U_{2,t}\rangle}.
\]
Note that $C_{\mathrm{IC}}$ is the uncentered cross-moment of the two realized IC series, not their Pearson correlation.
\par}

\Needspace*{8\baselineskip}
\section{The transverse statistic is partial correlation}\label{sec:3}

\begin{figure}[!tbp]
\centering
\adjustbox{max width=\linewidth}{
\begin{tikzpicture}[x=1cm,y=1cm,
  >={Stealth[length=2.2mm,width=1.8mm]},
  font=\footnotesize,
  navy/.style={color=fignavy},
  ret/.style={color=figred},
  vec/.style={line width=0.9pt,->},
  constr/.style={black!55,line width=0.5pt,dash pattern=on 2.2pt off 1.6pt},
  gauge/.style={black!50,line width=0.5pt,dash pattern=on 1.6pt off 1.6pt}]
  \node[anchor=north west,text=black!60,align=left] at (-6.3,5.75)
    {One date $t$; neutral subspace $H$, $\dim H=d-1$. Unit neutral signals $S_i$ in the moving frame:\\ $P_i=G_tS_i=p_ie_1+U_i$, where $G_t$ maps the realized return direction $Q_t=\widetilde R_t/a_t$ to $e_1$.};
  \begin{scope}[shift={(0.5,0)}]
  \fill[fignavy,opacity=0.06] (2.45,4.34) -- (-2.45,1.25) -- (-2.45,-4.34) -- (2.45,-1.25) -- cycle;
  \draw[fignavy!60,line width=0.5pt] (2.45,4.34) -- (-2.45,1.25) -- (-2.45,-4.34) -- (2.45,-1.25) -- cycle;
  \node[anchor=south east,navy] at (2.45,4.45) {$e_1^{\perp}$: transverse complement, $\dim=m=d-2$};
  \draw[figred!40,line width=0.6pt] (-3.65,0) -- (0,0);
  \draw[figred,line width=0.8pt,->] (0,0) -- (6.25,0);
  \node[anchor=north east,ret,align=right] at (6.25,-0.14) {$Q_t=e_1$\\ \textcolor{black!60}{realized return direction}};
  \begin{scope}[rotate around={-15:(-0.62,-1.23)}]
    \draw[gauge,->] (-0.62,-1.23) ++(205:0.45 and 0.6) arc (205:-115:0.45 and 0.6);
  \end{scope}
  \draw[black!50,line width=0.4pt] (-0.39,-1.85) -- (1.05,-2.95);
  \node[anchor=north west,text=black!60,align=left] at (1.1,-2.85)
    {residual gauge $O(m)$, $m=d-2$:\\ rotating the transverse complement\\ leaves only the Gram matrix\\ $\langle U_i,U_j\rangle$ invariant};
  \draw[constr] (1.38,2.47) -- (2.01,0);
  \draw[constr] (-0.63,2.47) -- (1.38,2.47);
  \draw[constr] (2.60,3.46) -- (1.28,0);
  \draw[constr] (1.32,3.46) -- (2.60,3.46);
  \draw[fignavy,line width=0.6pt] (69.1:1.05) arc (69.1:104.3:1.05);
  \node[navy] at (87:1.32) {$\theta$};
  \draw[vec,fignavy] (0,0) -- (-0.63,2.47);
  \draw[vec,fignavy] (0,0) -- (1.32,3.46);
  \node[anchor=east,navy] at (-0.72,2.55) {$U_1$};
  \node[anchor=south east,navy] at (1.30,3.52) {$U_2$};
  \draw[vec,black] (0,0) -- (1.38,2.47);
  \draw[vec,black] (0,0) -- (2.60,3.46);
  \node[anchor=west] at (1.45,2.55) {$P_1$};
  \node[anchor=west] at (2.68,3.52) {$P_2$};
  \draw[figred,line width=0.7pt] (2.01,-0.11) -- (2.01,0.11);
  \draw[figred,line width=0.7pt] (1.28,-0.11) -- (1.28,0.11);
  \node[anchor=north,ret] at (2.01,-0.14) {$p_1$};
  \node[anchor=north,ret] at (1.28,-0.14) {$p_2$};
  \fill (0,0) circle (1.1pt);
  \node[anchor=north east,text=black!60] at (-0.05,-0.08) {$0$};
  \end{scope}
  \node[anchor=north west,align=left] at (-6.3,4.35)
    {$\langle P_1,P_2\rangle=\langle S_1,S_2\rangle=p_1p_2+\langle U_1,U_2\rangle$\\[2pt]
     $p_i=\langle S_i,Q_t\rangle=$ realized IC of $S_i$,\quad $\lVert U_i\rVert_2^2=1-p_i^2$\\[2pt]
     $Z_t=\dfrac{\langle U_1,U_2\rangle}{\sqrt{(1-p_1^2)(1-p_2^2)}}=\cos\theta$\\[1pt]
     \textcolor{black!60}{$=$ partial correlation of $S_1$ and $S_2$ given $R_t$}};
\end{tikzpicture}}
\caption{\textbf{Longitudinal–transverse decomposition at one date.} Drawn in the moving frame of \cref{sec:2}, in which $G_t$ sends the realized-return direction $Q_t$ to $e_1$ and each unit neutral signal becomes $P_{i,t}=G_tS_{i,t}=p_{i,t}e_1+U_{i,t}$; because $G_t$ is an isometry, every inner product shown is that of the original signals. The longitudinal coordinate $p_{i,t}=\langle S_{i,t},Q_t\rangle$ is the realized Pearson IC (\cref{lem:1}); $U_{i,t}\in e_1^\perp$ is the transverse component, with $\lVert U_{i,t}\rVert_2^2=1-p_{i,t}^2$ and $\langle S_{1,t},S_{2,t}\rangle=p_{1,t}p_{2,t}+\langle U_{1,t},U_{2,t}\rangle$ (\cref{lem:2}). The angle $\theta$ is between the transverse components, and $Z_t=\cos\theta$ is the partial correlation of the two signals given $R_t$ (\cref{prop:3}). The transverse complement has dimension $m=d-2$; the figure shows the two-dimensional slice spanned by $U_{1,t}$ and $U_{2,t}$. Rotating that complement by $O(m)$ leaves only the transverse Gram matrix invariant (\cref{prop:4}).}
\label{fig:1}
\end{figure}
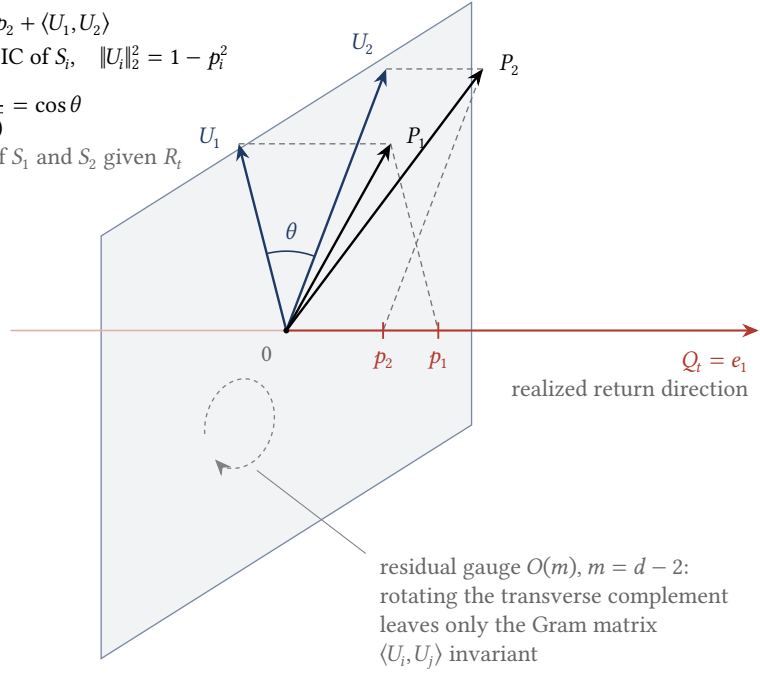

Normalize the transverse inner product by its attainable radius:

\[
Z_t
:=
\frac{\langle U_{1,t},U_{2,t}\rangle}
{\sqrt{(1-p_{1,t}^2)(1-p_{2,t}^2)}}.
\]

\Needspace*{10\baselineskip}
\begin{proposition}[transverse similarity equals partial correlation]\label{prop:3}
Whenever $|p_{1,t}|,|p_{2,t}|<1$,

\[
\boxed{
Z_t
=
\frac{
\rho_{12,t}^{\mathrm{sig}}-\rho_{1R,t}\rho_{2R,t}
}{
\sqrt{(1-\rho_{1R,t}^2)(1-\rho_{2R,t}^2)}
}
=
\rho_{12\cdot R,t}.}
\]

Normalized transverse similarity is therefore exactly the cross-sectional Pearson partial correlation between the two signals, controlling for the realized-return vector.

\end{proposition}

\Needspace*{4\baselineskip}
\begin{proof}
Substitute \cref{lem:1} and \cref{lem:2} into the standard one-control partial-correlation identity. Equivalently, regress each normalized signal cross-sectionally on $Q_t$. In the original coordinates the residuals are $G_t^{-1}U_{1,t}$ and $G_t^{-1}U_{2,t}$; because $G_t$ is orthogonal, their normalized inner product equals that of $U_{1,t}$ and $U_{2,t}$.
\end{proof}

Both coordinates now carry familiar statistical meanings:

\[
\boxed{\begin{aligned}
\text{longitudinal coordinate}&=\text{realized IC},\\
\text{normalized transverse similarity}&=\text{partial correlation given returns}.
\end{aligned}}
\]

\Needspace*{8\baselineskip}
\section{Residual-gauge classification}\label{sec:4}

Once $Q_t$ has been sent to $e_1$, the frame is still arbitrary inside $e_1^\perp$. The residual group is the stabilizer of $e_1$ in $O(H)$, namely $O(e_1^\perp)\cong O(m)$ with $m=d-2$, acting on all transverse vectors simultaneously.

\Needspace*{5\baselineskip}
\begin{proposition}[complete residual-gauge invariants]\label{prop:4}
At a fixed date, let $K$ unit signals have longitudinal vector

\[
p=(p_1,\ldots,p_K)^\top
\]

and transverse vectors $U_1,\ldots,U_K\in\mathbb R^m$. Two ordered configurations lie in the same residual $O(m)$ orbit if and only if they have the same $p$ and the same transverse Gram matrix

\[
B=(\langle U_i,U_j\rangle)_{i,j=1}^K.
\]

Every residual-gauge-invariant statistic therefore factors through $(p,B)$.

\end{proposition}

\Needspace*{4\baselineskip}
\begin{proof}
The forward direction is immediate. Conversely, suppose $(U_i)$ and $(U_i')$ share the Gram matrix $B$. If $\sum_i c_iU_i=0$, then

\[
c^\top Bc
=\left\lVert\sum_i c_iU_i\right\rVert_2^2
=0.
\]

The same Gram matrix gives

\[
\left\lVert\sum_i c_iU_i'\right\rVert_2^2
=c^\top Bc=0.
\]

So every linear dependence among the $U_i$ also holds among the $U_i'$, and the linear rule $U_i\mapsto U_i'$ is well defined. Equal Gram matrices make it an isometry between the two spans. Any isometry between subspaces of a finite-dimensional Euclidean space extends to an element of $O(m)$: extend orthonormal bases of the two spans to orthonormal bases of $\mathbb R^m$ and map one to the other. The two configurations therefore share an orbit.
\end{proof}

We use $O(H)$ rather than $SO(H)$ deliberately. Restricting to orientation-preserving frames can add an orientation invariant when the transverse vectors span the whole transverse space, which needs $K\ge m=d-2$. In the usual regime $K\ll d$ this cannot happen; allowing reflections discards the economically meaningless handedness convention in every case.

\Needspace*{8\baselineskip}
\section{PnL dependence and the centering of profitable signals}\label{sec:5}

Allocators do not size positions on signal correlation. They size on the covariance of realized payoffs, and the object they inspect is the Pearson correlation matrix of PnL series, typically on daily or weekly data and typically after each book has been scaled to a common volatility. Three things happen on the way from a signal to that matrix that do not happen in a signal correlation: each date is weighted by realized cross-sectional dispersion, the common expected-alpha component is centered out, and the result is normalized by time-series volatility rather than by cross-sectional norm. This section writes those three transformations out so that each can be attributed separately.

Define

\[
\mu_i=\overline{a_tp_{i,t}},
\qquad
\sigma_i^2=\overline{(a_tp_{i,t}-\mu_i)^2}.
\]

Then

\[
\boxed{
M_{\mathrm{pnl}}
:=\overline{\Pi_{1,t}\Pi_{2,t}}
=\overline{a_t^2p_{1,t}p_{2,t}},}
\]

and

\[
\boxed{
\rho_{\mathrm{pnl}}
=
\frac{
\overline{a_t^2p_{1,t}p_{2,t}}-\mu_1\mu_2
}{\sigma_1\sigma_2}.}
\]

Writing $x_t=p_{1,t}p_{2,t}$ separates pure scale from genuine dispersion-regime alignment:

\[
\boxed{
M_{\mathrm{pnl}}
=\overline{a^2}\,C_{\mathrm{IC}}
+\operatorname{Cov}_t(a_t^2,x_t).}
\]

The covariance term asks whether the two signals agree longitudinally precisely when realized cross-sectional opportunities are large.

Three distinct transformations are in play:

\begin{enumerate}
\item Signal similarity adds transverse partial association to longitudinal IC co-alignment.
\item The uncentered PnL second moment weights longitudinal IC products by $a_t^2$.
\item Pearson PnL correlation centers the dispersion-weighted IC series and normalizes their time-series volatility.
\end{enumerate}

\Needspace*{5\baselineskip}
\begin{remark}[Pearson centering removes shared expected alpha]\label{rem:5}
Write

\[
p_{i,t}=\alpha_i+\varepsilon_{i,t},
\qquad \overline{\varepsilon_i}=0.
\]

If $a_t=a$ is constant, then

\[
C_{\mathrm{IC}}
=\alpha_1\alpha_2+\overline{\varepsilon_{1,t}\varepsilon_{2,t}},
\]

but

\[
\rho_{\mathrm{pnl}}
=\operatorname{Corr}_t(\varepsilon_{1,t},\varepsilon_{2,t}).
\]

Two profitable signals can therefore share a positive expected IC yet have nearly uncorrelated PnL innovations. Pearson correlation strips out the shared mean-alpha component by design. How much this matters relative to transverse geometry and dispersion weighting is an empirical question.

\end{remark}

\Needspace*{8\baselineskip}
\section{Ensembles under Euclidean and economic risk budgets}\label{sec:6}

In practice no one trades a single signal at Euclidean norm one. Alphas are combined, the combination is passed through a risk model, and the book is scaled to a volatility or gross-exposure budget. Two different questions therefore arise. The first is statistical: how much of the realized return direction can the signal span explain, which is a question about ordinary Pearson IC and is answered in the Euclidean metric. The second is economic: how much payoff a combination delivers per unit of modelled risk, which is answered in the covariance metric of the risk model. Practitioners often speak as if these were the same question. This section treats them as two lenses on the same longitudinal vector and shows precisely when they agree.

For $K$ signals, collect their columns in $S_t\in\mathbb R^{d\times K}$, their Euclidean longitudinal coordinates in $p_t=S_t^\top Q_t$, and their transverse coordinates in $U_t$. The Euclidean signal Gram matrix is

\[
\boxed{
A_t:=S_t^\top S_t=p_tp_t^\top+U_t^\top U_t.}
\]

For ensemble weights $v\in\mathbb R^K$,

\[
\Pi_t(v)=a_tp_t^\top v.
\]

After Euclidean normalization, the realized Pearson IC of the combination is

\[
\boxed{
p_t^{\mathrm{comb}}(v)
=
\frac{p_t^\top v}{\sqrt{v^\top A_tv}}
=
\frac{p_t^\top v}
{\sqrt{v^\top p_tp_t^\top v+v^\top U_t^\top U_tv}}.}
\]

\Needspace*{5\baselineskip}
\begin{proposition}[optimal Euclidean-normalized combination]\label{prop:6}
If $A_t$ is positive definite, then

\[
\boxed{
\max_{v\ne0}\big|p_t^{\mathrm{comb}}(v)\big|^2
=p_t^\top A_t^{-1}p_t,
\qquad
v_t^\star\propto A_t^{-1}p_t.}
\]

\end{proposition}

\Needspace*{4\baselineskip}
\begin{proof}
Set $y=A_t^{1/2}v$. Cauchy–Schwarz gives

\[
\frac{(p_t^\top v)^2}{v^\top A_tv}
=
\frac{\big((A_t^{-1/2}p_t)^\top y\big)^2}{y^\top y}
\le p_t^\top A_t^{-1}p_t,
\]

with equality when $y\propto A_t^{-1/2}p_t$.
\end{proof}

The best ex post Pearson IC available from the signal span is the Mahalanobis norm of $p_t$ in the signal-Gram metric. Transverse structure enters exactly through $A_t^{-1}$.

If the signal columns are linearly dependent, replace $A_t^{-1}$ by the Moore–Penrose pseudoinverse $A_t^+$. The optimum stays finite: $v\in\ker(A_t)$ implies $S_tv=0$ and hence $p_t^\top v=Q_t^\top S_tv=0$, so $p_t\in\operatorname{range}(A_t)$.

For an equal-weight pair,

\[
p_t^{\mathrm{comb}}
=
\frac{(p_{1,t}+p_{2,t})/2}
{\sqrt{\left(1+p_{1,t}p_{2,t}+\langle U_{1,t},U_{2,t}\rangle\right)/2}}.
\]

With the individual longitudinal coordinates held fixed, greater transverse similarity inflates the combination norm and shrinks the normalized IC; transverse cancellation can raise IC per unit Euclidean norm.

\Needspace*{5\baselineskip}
\begin{corollary}[span optimum is the multiple correlation]\label{cor:6.1}
Let

\[
P_{S_t}=S_t(S_t^\top S_t)^+S_t^\top
\]

be the orthogonal projector onto the contemporaneous signal span. Then

\[
\boxed{
p_t^\top A_t^+p_t
=Q_t^\top P_{S_t}Q_t
=\lVert P_{S_t}Q_t\rVert_2^2
=\mathcal R_t^2,}
\]

where $\mathcal R_t^2$ is the cross-sectional coefficient of determination from regressing the realized return direction on the signal span. So

\[
\sqrt{p_t^\top A_t^+p_t}
\]

is the multiple correlation coefficient between realized returns and that span: the natural multi-signal extension of single-signal IC.

\end{corollary}

The same fact limits its interpretation. If a fixed rank-$K$ signal span is independent of an isotropic $Q_t$ in the $(d-1)$-dimensional neutral space, then

\[
\mathcal R_t^2
\sim
\operatorname{Beta}\!\left(\frac K2,\frac{d-1-K}{2}\right),
\qquad
\mathbb E[\mathcal R_t^2]=\frac{K}{d-1}.
\]

Raw ex post $R^2$ therefore rises mechanically with library rank even under a no-information null. A null-centered, or adjusted, statistic is

\[
\boxed{
\mathcal R_{\mathrm{adj},t}^2
=1-(1-\mathcal R_t^2)\frac{d-1}{d-1-K},}
\]

which has null expectation zero when the stated model applies. Neither raw nor adjusted ex post $R^2$ is a trading objective; both measure contemporaneous span capacity.

\Needspace*{5\baselineskip}
\begin{proposition}[optimal combination under a general risk metric]\label{prop:7}
Let $\Sigma_t$ be positive definite on the tradable neutral subspace and define

\[
B_t:=S_t^\top\Sigma_tS_t.
\]

The realized payoff per unit ex ante risk is

\[
\mathcal E_t(v)
:=
\frac{\Pi_t(v)}{\sqrt{v^\top B_tv}}
=
a_t\frac{p_t^\top v}{\sqrt{v^\top B_tv}}.
\]

If $B_t$ is positive definite, then

\[
\boxed{
\max_{v\ne0}|\mathcal E_t(v)|^2
=a_t^2p_t^\top B_t^{-1}p_t,
\qquad
v_{t,\Sigma}^\star\propto B_t^{-1}p_t.}
\]

\end{proposition}

\begin{proof}
The proof is that of \cref{prop:6} with $A_t$ replaced by $B_t$.
\end{proof}

The same pseudoinverse extension covers singular $B_t$ when $\Sigma_t$ is positive definite on the tradable subspace.

Both optimizers are ex post envelopes, because $p_t$ contains the realized forward return. They are not trading rules. A feasible strategy must estimate the relevant longitudinal vector, or its conditional expectation, from information available before the return is realized, regularize the Gram inverse, and validate the weights out of sample. The ex post optimum remains useful as a span-efficiency diagnostic and theoretical upper bound.

Optimal weights of the form $C^{-1}E$, an inverse covariance applied to a vector of expected alphas, are standard for combining a large alpha library \parencite{kakushadze2017billion}. The objects differ from those here: that covariance is a time-series covariance of alpha returns, whereas $B_t$ is a single-date cross-sectional risk Gram matrix built from the signal columns. The present optimum is therefore a contemporaneous span-efficiency envelope rather than an allocation rule estimated over history.

The mixed-metric formulation of \cref{prop:7} keeps $p_t^\top v=\Pi_t(v)/a_t$ as the longitudinal exposure of the unnormalized combination while putting the economically relevant covariance metric in the denominator. It should not be called the IC of the combination. Alternatively, define

\[
b_t=\sqrt{\widetilde R_t^\top\Sigma_t^{-1}\widetilde R_t},
\qquad
Q_t^{\Sigma}
=\frac{\Sigma_t^{-1}\widetilde R_t}{b_t}.
\]

Then $Q_t^\Sigma$ is unit under $\langle x,y\rangle_\Sigma=x^\top\Sigma_ty$, and

\[
\gamma_{i,t}^\Sigma
:=
\frac{\langle S_{i,t},Q_t^\Sigma\rangle_\Sigma}
{\lVert S_{i,t}\rVert_\Sigma}
=\frac{S_{i,t}^\top\widetilde R_t}
{b_t\lVert S_{i,t}\rVert_\Sigma}.
\]

The whole longitudinal–transverse construction can be repeated in this risk geometry. The normalized coordinate $\gamma_{i,t}^\Sigma$ is a genuine risk-metric cosine, and the normalized transverse statistic is a $\Sigma_t$-partial correlation—not ordinary Pearson IC or partial correlation. Distribution theory for risk-adjusted cross-sectional ICs of this kind is developed in \textcite{ding2023timevarying}.

The paper therefore carries two linked lenses:

\begin{itemize}
\item \textbf{Pearson lens:} $A_t=S_t^\top S_t$, giving literal IC and ordinary partial correlation.
\item \textbf{Risk lens:} $B_t=S_t^\top\Sigma_tS_t$, giving payoff per unit economic risk.
\end{itemize}

The two rules agree for every longitudinal vector $p$ if and only if $B_t\propto A_t$, that is, if and only if the restriction of the quadratic form of $\Sigma_t$ to the signal span is a scalar multiple of the Euclidean form; this does not require $\Sigma_t$ to map the span into itself. We use \emph{isotropic on the signal span} as shorthand for this condition. For a particular realized $p_t$ the two optima can coincide under weaker conditions, for instance whenever $p_t$ is an eigenvector of $B_tA_t^{-1}$. A real implementation should report both; statistical similarity and economic risk are not the same metric.

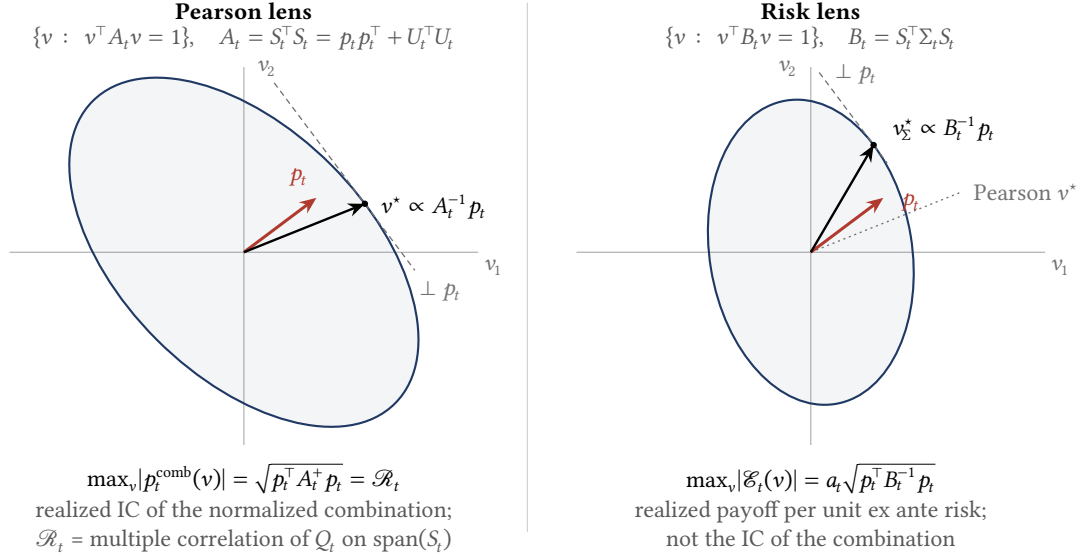
\begin{figure}[!tbp]
\centering
\adjustbox{max width=\linewidth}{
\begin{tikzpicture}[x=1cm,y=1cm,
  >={Stealth[length=2.2mm,width=1.8mm]},
  font=\footnotesize,
  axis/.style={black!35,line width=0.4pt},
  level/.style={fignavy,line width=0.8pt},
  support/.style={black!55,line width=0.5pt,dash pattern=on 2.2pt off 1.6pt},
  vec/.style={line width=0.9pt,->}]
  \begin{scope}[shift={(-3.75,0)}]
    \node[anchor=south,font=\footnotesize\bfseries] at (0,2.95) {Pearson lens};
    \node[anchor=south,text=black!65] at (0,2.55) {$\{v:\ v^\top A_tv=1\}$,\quad $A_t=S_t^\top S_t=p_tp_t^\top+U_t^\top U_t$};
    \draw[axis] (-3.1,0) -- (3.1,0);  \draw[axis] (0,-2.55) -- (0,2.45);
    \node[anchor=west,text=black!55] at (3.05,-0.2) {$v_1$};
    \node[anchor=south west,text=black!55] at (0.05,2.2) {$v_2$};
    \fill[fignavy,opacity=0.05,rotate=45] (0,0) ellipse (1.633 and 2.828);
    \draw[level,rotate=45] (0,0) ellipse (1.633 and 2.828);
    \draw[support] ($(1.601,0.641)+(-1.2,1.6)$) -- ($(1.601,0.641)+(0.66,-0.88)$);
    \node[anchor=north west,text=black!55] at (2.2,-0.25) {$\perp p_t$};
    \draw[vec,figred,line width=1.15pt] (0,0) -- (0.96,0.72);
    \node[anchor=south east,text=figred] at (0.98,0.74) {$p_t$};
    \draw[vec,black] (0,0) -- (1.601,0.641);
    \fill (1.601,0.641) circle (1.2pt);
    \node[anchor=west] at (1.68,0.60) {$v^\star\propto A_t^{-1}p_t$};
    \node[anchor=north] at (0,-2.65) {$\max_v\lvert p_t^{\mathrm{comb}}(v)\rvert=\sqrt{p_t^\top A_t^{+}p_t}=\mathcal R_t$};
    \node[anchor=north,text=black!65,align=center] at (0,-3.15) {realized IC of the normalized combination;\\ $\mathcal R_t$ = multiple correlation of $Q_t$ on $\operatorname{span}(S_t)$};
  \end{scope}
  \draw[black!20,line width=0.4pt] (0,3.3) -- (0,-3.7);
  \begin{scope}[shift={(3.75,0)}]
    \node[anchor=south,font=\footnotesize\bfseries] at (0,2.95) {Risk lens};
    \node[anchor=south,text=black!65] at (0,2.55) {$\{v:\ v^\top B_tv=1\}$,\quad $B_t=S_t^\top\Sigma_tS_t$};
    \draw[axis] (-3.1,0) -- (3.1,0);  \draw[axis] (0,-2.55) -- (0,2.45);
    \node[anchor=west,text=black!55] at (3.05,-0.2) {$v_1$};
    \node[anchor=south east,text=black!55] at (-0.05,2.2) {$v_2$};
    \fill[fignavy,opacity=0.05,rotate=9.22] (0,0) ellipse (1.339 and 2.033);
    \draw[level,rotate=9.22] (0,0) ellipse (1.339 and 2.033);
    \draw[black!55,line width=0.5pt,dash pattern=on 0.8pt off 1.4pt] (0,0) -- (2.0,0.80);
    \node[anchor=west,text=black!55] at (2.02,0.80) {Pearson $v^\star$};
    \draw[support] ($(0.831,1.418)+(-0.72,0.96)$) -- ($(0.831,1.418)+(0.18,-0.24)$);
    \node[anchor=west,text=black!55] at (0.20,2.36) {$\perp p_t$};
    \draw[vec,figred,line width=1.15pt] (0,0) -- (0.96,0.72);
    \node[anchor=west,text=figred] at (1.06,0.66) {$p_t$};
    \draw[vec,black] (0,0) -- (0.831,1.418);
    \fill (0.831,1.418) circle (1.2pt);
    \node[anchor=west] at (0.95,1.64) {$v^\star_{\Sigma}\propto B_t^{-1}p_t$};
    \node[anchor=north] at (0,-2.65) {$\max_v\lvert\mathcal E_t(v)\rvert=a_t\sqrt{p_t^\top B_t^{-1}p_t}$};
    \node[anchor=north,text=black!65,align=center] at (0,-3.15) {realized payoff per unit ex ante risk;\\ not the IC of the combination};
  \end{scope}
  \node[anchor=north,align=center,text=black!75] at (0,-4.0)
    {The two rules agree for every $p_t$ iff $\Sigma_t$ is isotropic on the signal span; for one $p_t$ they can coincide by chance.\\
     Both are ex post envelopes ($p_t$ contains the realized return), not trading rules.};
\end{tikzpicture}}
\caption{\textbf{The same longitudinal vector, two metrics.} Weight space $v\in\mathbb R^K$ for $K=2$ signals. \emph{Left (Pearson lens):} the unit level set of the Euclidean signal Gram metric $A_t=S_t^\top S_t=p_tp_t^\top+U_t^\top U_t$. Maximizing the longitudinal exposure $p_t^\top v$ over this set gives $v_t^\star\propto A_t^{-1}p_t$, the point where a supporting line perpendicular to $p_t$ touches the ellipse; the attained value is the realized Pearson IC of the normalized combination, $\max_v|p_t^{\mathrm{comb}}(v)|=\sqrt{p_t^\top A_t^{+}p_t}=\mathcal R_t$, the multiple correlation of $Q_t$ on the signal span (\cref{prop:6}, \cref{cor:6.1}). \emph{Right (risk lens):} the unit level set of $B_t=S_t^\top\Sigma_tS_t$. The same construction gives $v_{t,\Sigma}^\star\propto B_t^{-1}p_t$ and $\max_v|\mathcal E_t(v)|=a_t\sqrt{p_t^\top B_t^{-1}p_t}$, the realized payoff per unit ex ante risk (\cref{prop:7}); this is not the IC of the combination. The dotted ray reproduces the Pearson optimum for comparison. The two rules coincide for every $p_t$ if and only if $\Sigma_t$ is isotropic on the signal span; for a particular $p_t$ they can coincide under weaker conditions. Both are ex post envelopes—$p_t$ contains the realized forward return—and measure contemporaneous span capacity, not deployable alpha.}
\label{fig:2}
\end{figure}

\Needspace*{8\baselineskip}
\section{Sharp bounds and non-identifiability}\label{sec:7}

\begin{proposition}[sharp pointwise bounds]\label{prop:8}
For neutral unit signals,

\[
\boxed{
\left|
\rho_{12,t}^{\mathrm{sig}}-p_{1,t}p_{2,t}
\right|
\le
\sqrt{(1-p_{1,t}^2)(1-p_{2,t}^2)}.}
\]

The endpoints correspond to transverse partial correlations $Z_t=\pm1$.

\end{proposition}

At a realized IC of $0.03$,

\[
\lVert U_{i,t}\rVert_2\approx0.99955,
\]

so $99.91\%$ of squared signal norm is transverse to that date's realized return. That component is not economically inert in general; it is simply orthogonal to that fixed signal's same-date payoff direction.

\Needspace*{5\baselineskip}
\begin{theorem}[no universal control between signal and PnL correlation]\label{thm:9}
Assume $d\ge4$, so $m=d-2\ge2$. Without restrictions on transverse geometry, return magnitudes, or temporal dependence, neither average signal similarity nor PnL correlation determines, nontrivially bounds, or universally orders the other.

\end{theorem}

\Needspace*{4\baselineskip}
\begin{proof}
For the first direction, take $a_t\equiv1$ and any nonconstant $p_t\in(-1,1)$. Set

\[
P_{1,t}=p_te_1+\sqrt{1-p_t^2}V_t,
\qquad
P_{2,t}=p_te_1+\sqrt{1-p_t^2}W_t,
\]

where $V_t,W_t$ are transverse unit vectors. The two PnLs are identical, so $\rho_{\mathrm{pnl}}=1$, while

\[
\langle P_{1,t},P_{2,t}\rangle
=p_t^2+(1-p_t^2)\langle V_t,W_t\rangle.
\]

Because $m\ge2$, the final inner product can take any value in $[-1,1]$.

For the converse, choose two pairs of nonconstant longitudinal series, each bounded in magnitude by $\varepsilon$, with different time-series correlations—say zero and $0.9$. Fix a constant target signal similarity $c$. The transverse choice

\[
\langle V_t,W_t\rangle
=
\frac{c-p_{1,t}p_{2,t}}
{\sqrt{(1-p_{1,t}^2)(1-p_{2,t}^2)}}
\]

is feasible whenever

\[
|c-p_{1,t}p_{2,t}|
\le
\sqrt{(1-p_{1,t}^2)(1-p_{2,t}^2)}.
\]

A sufficient uniform condition is $|c|\le1-2\varepsilon^2$: the left side is at most $|c|+\varepsilon^2\le1-\varepsilon^2$, while the right side is at least $1-\varepsilon^2$. Both longitudinal pairs can therefore be completed by transverse vectors that produce exactly the same $c$ at every date but different PnL correlations.

Finally, these constructions are genuine signals, not merely formal coordinates: applying $G_t^{-1}$ to each $P_{i,t}$ along any chosen return path $\widetilde R_t=a_tQ_t$ yields neutral unit signal vectors in the original frame.
\end{proof}

When $d=3$, the transverse space is one-dimensional and only the two choices $Z_t=\pm1$ are available; a weaker discrete non-identifiability remains.

\Needspace*{8\baselineskip}
\section{Null distribution of transverse partial correlation}\label{sec:8}

A measured transverse partial correlation is only meaningful against a scale. In a $d$-asset cross-section, two signals that share no structure whatsoever will still exhibit a nonzero residual similarity on any given date, simply because two random directions in a finite-dimensional space are never exactly orthogonal. The isotropic null below is that scale. It is not proposed as a model of how signals are generated—signals built from a common feature pipeline are anything but isotropic—but as the exact reference against which persistent transverse structure can be quantified, in the same way that a $t$-distribution is a reference for a regression coefficient rather than a belief about the data.

\Needspace*{5\baselineskip}
\begin{proposition}[exact null law]\label{prop:10}
Conditional on $p_{1,t},p_{2,t}$, suppose the normalized transverse directions are independent and uniform on $S^{m-1}$, the unit sphere of the $m$-dimensional transverse space $e_1^\perp$ \parencite{mardia2000directional}. Then $Z_t$ has density

\[
f_m(z)
=
\frac{\Gamma(m/2)}{\sqrt\pi\,\Gamma((m-1)/2)}
(1-z^2)^{(m-3)/2},
\qquad -1<z<1,
\]

with

\[
\mathbb E[Z_t]=0,
\qquad
\operatorname{Var}(Z_t)=\frac1m.
\]

\end{proposition}

Independent jointly Gaussian cross-sectional observations, after demeaning and partialling out one control, induce the same independent uniform residual directions. The geometric isotropic null and the classical Gaussian partial-correlation null are therefore observationally equivalent for $Z_t$, although Gaussianity is a stronger assumption on the underlying data. Under either construction \parencite{fisher1924partial,anderson2003multivariate},

\[
\boxed{
\tau_t
=
\frac{Z_t\sqrt{m-1}}{\sqrt{1-Z_t^2}}
\sim t_{m-1}=t_{d-3}.}
\]

Fisher's transform \parencite{fisher1921probable} gives the large-$d$ approximation

\[
\operatorname{arctanh}(Z_t)
\approx
\mathcal N\!\left(0,\frac1{m-2}\right).
\]

The degrees of freedom double as a dimension check: demeaning uses one of the $d$ asset dimensions and partialling out $Q_t$ uses another, leaving $m=d-2$ transverse dimensions; the classical statistic has $d-3=m-1$ degrees of freedom, and the Fisher-$z$ variance uses $d-4=m-2$, one fewer, as for any partial correlation with one control.

The isotropic/Gaussian null is not a realistic generative model for signals that come out of a shared feature pipeline. Its value is as an exact finite-dimensional reference scale: systematic departures quantify persistent transverse structure; they do not prove that such structure “should not” exist.

Across dates, serial dependence and overlapping return horizons rule out an independent-date aggregation. Fisher-transformed partial correlations can instead be aggregated with a dependence-aware block bootstrap or a HAC procedure.

\Needspace*{8\baselineskip}
\section{Synthetic illustration}\label{sec:9}

The construction below shows a conditional claim at realistic weak-IC scales: \textbf{if transverse feature-family structure is designed independently of longitudinal IC dependence, signal similarity need not rank PnL dependence correctly.} The independence is imposed, not discovered; whether it holds approximately in a real signal library is an empirical question. This is a mechanism illustration, not calibrated evidence.

Use seed $20260906$, $d=500$, $T=1000$, and $K=12$. Under a zero-alpha isotropic benchmark, a single-date IC has standard deviation approximately $1/\sqrt{d-1}=0.0448$. We therefore use an IC innovation scale of $0.045$ rather than make the IC paths artificially smoother than their cross-sectional noise floor. Let mean ICs be equally spaced from $0.015$ to $0.025$ and generate

\[
p_{i,t}
=\alpha_i
+0.045\left(\ell_if_t+\sqrt{1-\ell_i^2}\,\varepsilon_{i,t}\right),
\]

where $f_t$ and $\varepsilon_{i,t}$ are independent standard normals and $\ell_i$ is equally spaced from $-0.7$ to $0.7$. Generate dispersion by

\[
a_t
=\frac{\exp(0.35|f_t|+0.25\eta_t)}
{\overline{\exp(0.35|f|+0.25\eta)}},
\]

with independent standard-normal $\eta_t$. Assign signals cyclically to three feature families. At each date draw a unit Gaussian direction for every family and an independent unit Gaussian direction for every signal in the $m=498$ dimensional transverse space. Mix family and idiosyncratic directions using loadings $\beta_i$ repeating $(0.15,0.35,0.60,0.80)$, renormalize the mixture, and scale it by $\sqrt{1-p_{i,t}^2}$.

Across the resulting $66$ signal pairs:

\par\vspace{\baselineskip}\noindent\begin{minipage}{\linewidth}\centering\small
\captionof{table}{Diagnostics across the 66 signal pairs of the synthetic design (seed 20260906, $d=500$, $T=1000$, $K=12$).}\label{tab:1}
\begin{tabular}{lr}
\toprule
Diagnostic & Simulated value \\
\midrule
Mean-IC range & $0.0146$ to $0.0277$ \\
Spearman rank correlation: $C_{\mathrm{sig}}$ vs. $\rho_{\mathrm{pnl}}$ & $-0.070$ \\
Spearman rank correlation: $C_{\mathrm{IC}}$ vs. $\rho_{\mathrm{pnl}}$ & $0.974$ \\
Mean $\mathcal R_t^2$ & $0.0334$ \\
Isotropic-null mean $K/(d-1)$ & $0.0240$ \\
Mean $\mathcal R_{\mathrm{adj},t}^2$ & $0.0096$ \\
Same-family median $|C_\perp|$ & $0.1641$ \\
Same-family 5\%–95\% range of $C_\perp$ & $0.0524$ to $0.4781$ \\
Cross-family median $|C_\perp|$ & $0.00094$ \\
Cross-family 5\%–95\% range of $C_\perp$ & $-0.00265$ to $0.00171$ \\
\bottomrule
\end{tabular}
\end{minipage}\par\vspace{\baselineskip}

Selected pairs expose the mechanism:

\par\vspace{\baselineskip}\noindent\begin{minipage}{\linewidth}\centering\small
\captionof{table}{Selected pairs from the synthetic design.}\label{tab:2}
\begin{tabular}{lrrrrr}
\toprule
Pair & $C_{\mathrm{sig}}$ & $C_{\mathrm{IC}}$ & $C_\perp$ & $\rho_{\mathrm{pnl}}$ & $\operatorname{Cov}(a^2,p_1p_2)$ \\
\midrule
3–12 & $0.4766$ & $-0.00016$ & $0.4767$ & $-0.405$ & $-0.000563$ \\
10–11 & $-0.0011$ & $0.00115$ & $-0.0022$ & $0.427$ & $0.000662$ \\
2–11 & $0.2091$ & $-0.00033$ & $0.2095$ & $-0.539$ & $-0.000949$ \\
2–3 & $-0.0014$ & $0.00085$ & $-0.0022$ & $0.415$ & $0.000439$ \\
11–12 & $-0.0002$ & $0.00168$ & $-0.0019$ & $0.616$ & $0.000981$ \\
\bottomrule
\end{tabular}
\end{minipage}\par\vspace{\baselineskip}

Under the imposed independence, then, IC co-alignment ranks PnL dependence almost perfectly, while full signal similarity carries almost none of that ranking information. Two signals can look nearly orthogonal while their PnLs are strongly correlated, or look strongly similar through a shared transverse feature family while their PnLs are negatively correlated. The sharply different same-family and cross-family distributions also show why the overall median is not an informative summary.

The qualitative direction of this result is not new: \textcite{sorensen2004multiple} reach the same conclusion by simulation, finding that the time-series correlation of information coefficients matters more for diversification than the cross-sectional correlation of the signals themselves. What the decomposition adds is an exact account of the difference, which allows the gap to be attributed to specific components rather than observed as an aggregate.

The dispersion-regime term is not negligible in this design: across pairs, its median absolute magnitude is roughly one-half of the pure scale term $\overline{a^2}C_{\mathrm{IC}}$. Its sign separates signals that agree particularly in high-dispersion states from those whose agreement is concentrated in quieter states.

As a sensitivity check, rerun the otherwise identical design at four IC innovation scales:

\par\vspace{\baselineskip}\noindent\begin{minipage}{\linewidth}\centering\small
\captionof{table}{Sensitivity of the synthetic design to the IC innovation scale.}\label{tab:3}
\begin{tabular}{rrrrr}
\toprule
IC innovation sd & Spearman $C_{\mathrm{sig}}$ vs. $\rho_{\mathrm{pnl}}$ & Spearman $C_{\mathrm{IC}}$ vs. $\rho_{\mathrm{pnl}}$ & Mean $\mathcal R^2$ & Mean $\mathcal R_{\mathrm{adj}}^2$ \\
\midrule
$0.012$ & $-0.165$ & $0.388$ & $0.0055$ & $-0.0190$ \\
$0.025$ & $-0.160$ & $0.877$ & $0.0128$ & $-0.0115$ \\
$0.045$ & $-0.070$ & $0.974$ & $0.0334$ & $0.0096$ \\
$0.070$ & $0.150$ & $0.989$ & $0.0729$ & $0.0501$ \\
\bottomrule
\end{tabular}
\end{minipage}\par\vspace{\baselineskip}

Signal correlation stays a poor ranking proxy through the natural $0.045$ noise scale and turns mildly positive only well beyond it. The multiple-correlation diagnostics also show why the original $0.012$ calibration was unsatisfactory: its signal span covered less of the realized-return direction than the isotropic rank-$12$ null, because the prescribed per-date ICs were smaller than random orientation would ordinarily produce—not because the span lost rank.

The experiment should eventually be replaced or supplemented by an out-of-sample study on a real signal library, with the inference and multiplicity controls below.

\Needspace*{8\baselineskip}
\section{Non-neutral signals and risk geometry}\label{sec:10}

Let $E_0=\mathbf1/\sqrt d$ and drop neutrality while keeping $\lVert S_{i,t}\rVert_2=1$. Decompose

\[
S_{i,t}=c_{i,t}E_0+p_{i,t}Q_t+U_{i,t},
\]

where $U_{i,t}\perp E_0,Q_t$. Then

\[
\boxed{
c_{i,t}^2+p_{i,t}^2+\lVert U_{i,t}\rVert_2^2=1,}
\]

and

\[
\boxed{
\Pi_{i,t}=\sqrt d\,\bar R_tc_{i,t}+a_tp_{i,t}.}
\]

Consequently,

\[
\begin{aligned}
\overline{\Pi_{1,t}\Pi_{2,t}}
={}&
\overline{a_t^2p_{1,t}p_{2,t}}
+d\,\overline{\bar R_t^2c_{1,t}c_{2,t}}\\
&+\sqrt d\,\overline{
a_t\bar R_t
(p_{1,t}c_{2,t}+p_{2,t}c_{1,t})}.
\end{aligned}
\]

The extra channels are the common-market component and two market–cross-sectional interactions. The neutral sharp bound likewise has $1-p_{i,t}^2$ replaced by $1-c_{i,t}^2-p_{i,t}^2$.

The partial-correlation reading survives as well. Partialling out $E_0$ is cross-sectional demeaning; partialling out $Q_t$ then removes the realized-return direction. Therefore

\[
\boxed{
Z_t^{\mathrm{nonneutral}}
=
\frac{
\langle S_{1,t},S_{2,t}\rangle-c_{1,t}c_{2,t}-p_{1,t}p_{2,t}
}{
\sqrt{(1-c_{1,t}^2-p_{1,t}^2)(1-c_{2,t}^2-p_{2,t}^2)}
}}
\]

is the cross-sectional partial correlation of the two raw signals controlling for the span of $\{\mathbf1,R_t\}$. The transverse dimension is still $d-2$, and the classical one-return-control partial-correlation test keeps $d-3$ degrees of freedom because demeaning already accounts for the intercept.

Transversality is always relative to realized returns at the stated date and horizon. A transverse component may align with later returns, or load on ex ante market, sector, style, and liquidity factors. Orthogonality to realized returns and orthogonality to a risk-model span are different conditions.

This paper does not derive turnover or transaction costs from the transverse decomposition. Turnover depends on $\lVert S_{i,t}-S_{i,t-1}\rVert$—a within-signal comparison across dates—whereas the present Gram geometry compares signals at one date. The target-aligned frame also rotates between $t-1$ and $t$, so longitudinal and transverse coordinates cannot be differenced without a rule that transports one frame to the next. Building such a temporal connection is a separate extension; cost-aware implementation enters here only through the distinction between research signals and implemented portfolios.

\Needspace*{8\baselineskip}
\section{Estimation and multiple comparisons}\label{sec:11}

A library of even modest size generates a large number of pairwise diagnostics: fifty signals yield 1,225 pairs, each with its own $C_{\mathrm{sig}}$, $C_{\mathrm{IC}}$, $C_\perp$, and $\rho_{\mathrm{pnl}}$. Ranking those pairs, declaring some \enquote{diversifying,} or flagging others as sharing transverse structure are all multiple-comparison problems, and the daily observations underlying each estimate are serially dependent through signal persistence and overlapping forward-return windows. The estimation recommendations here are therefore conservative by design, and are intended to be read alongside the literature on data snooping in alpha research cited in \cref{sec:1.1}.

For every signal pair, estimate

\[
C_{\mathrm{sig}},\quad
C_{\mathrm{IC}},\quad
C_\perp,\quad
\overline{a_t^2p_{1,t}p_{2,t}},\quad
\rho_{\mathrm{pnl}},\quad
\text{and}\quad
\overline{\operatorname{arctanh}(Z_t)}.
\]

Inference should keep the date-level joint contributions, because $\widehat C_\perp=\widehat C_{\mathrm{sig}}-\widehat C_{\mathrm{IC}}$ inherits their covariance. Moving-block or stationary bootstraps \parencite{kunsch1989jackknife,politis1994stationary,lahiri2003resampling} should reflect signal persistence and overlapping forward-return horizons; HAC inference \parencite{newey1987simple} is an alternative for smooth moment estimators.

A library of $K$ signals has $K(K-1)/2$ pairs. Pairwise discovery, rank comparisons, and isotropy rejections therefore need multiplicity control. Depending on the claim, use false-discovery-rate control \parencite{benjamini1995fdr}, family-wise bootstrap procedures \parencite{white2000reality,romano2005stepwise}, or simultaneous confidence intervals. A point-estimate ordering is not evidence of a stable ordering.

There is no universal sample-size threshold. Power depends on cross-sectional dimension, serial dependence, horizon overlap, return-dispersion variation, and the separation between competing pairwise statistics.

\Needspace*{8\baselineskip}
\section{Central result and empirical program}\label{sec:12}

The exact hierarchy is

\begingroup\small
\[
\boxed{
\begin{aligned}
\text{signal similarity}
&=\text{uncentered IC co-alignment}
+\text{transverse residual similarity},\\
\text{normalized transverse similarity}
&=\text{partial correlation controlling for realized returns},\\
\text{uncentered PnL second moment}
&=\text{dispersion-weighted uncentered IC cross-moment},\\
\text{PnL correlation}
&=\text{centered, normalized dependence of dispersion-weighted ICs}.
\end{aligned}}
\]
\endgroup

The separate ensemble results are

\[
\boxed{
\begin{aligned}
\text{squared multiple correlation}
&=p_t^\top(S_t^\top S_t)^+p_t
=\mathcal R_t^2,\\
\text{risk-optimal ex post weights}
&\propto(S_t^\top\Sigma_tS_t)^+p_t.
\end{aligned}}
\]

These are ex post span-fit statements. Raw $\mathcal R_t^2$ rewards increasing signal rank under the null; the risk expression is unsigned after squaring, contains realized information through $p_t$, and must keep the $a_t^2$ factor when read as squared payoff efficiency. Neither is a deployable library score without null adjustment, sign discipline, regularization, and out-of-sample estimation.

This gives a practical research program:

\begin{enumerate}
\item Decompose observed signal correlation into realized-IC co-alignment and partial correlation given realized returns.
\item Attribute disagreement between signal and PnL correlation to transverse geometry, dispersion weighting, implementation scaling, and Pearson centering.
\item Compare transverse partial correlations with their isotropic reference scale using dependence-aware inference.
\item Identify the feature families, construction pipelines, or risk exposures behind persistent transverse Gram structure.
\item Use $\mathcal R_{\mathrm{adj},t}^2$ only as a null-relative diagnostic of Euclidean span coverage. For economic selection, estimate a signed longitudinal target before realization, use a regularized risk Gram inverse, and evaluate the resulting realized payoff per unit risk strictly out of sample against signal-correlation baselines.
\end{enumerate}

The result does not claim that an elementary projection formula is new. It shows that familiar practitioner objects—signal correlation, realized IC, partial correlation, transfer efficiency, and PnL correlation—are coordinates and transformations of a single geometric decomposition. That decomposition says exactly what each diversity measure sees, what it discards, and where the discarded geometry re-enters once signals are combined into a constrained book.

\begin{practicebox}{What this means in practice}
\begin{itemize}[leftmargin=1.2em,itemsep=0.35em]
\item \textbf{Signal correlation and PnL correlation measure different things, and neither bounds the other.} Signal correlation adds two components: how the signals' realized ICs co-move, and how similar they are in the directions orthogonal to that date's returns. PnL correlation sees only the first, reweighted by dispersion and centered. A low signal correlation is not a diversification guarantee, and a high one is not a diversification failure (\cref{thm:9}).
\item \textbf{Decompose before you decide.} For each pair, report $C_{\mathrm{sig}}$, $C_{\mathrm{IC}}$, and $C_\perp=C_{\mathrm{sig}}-C_{\mathrm{IC}}$ together. Disagreement between signal and PnL correlation can then be attributed to transverse geometry, dispersion weighting, implementation scaling, or Pearson centering, rather than dismissed as noise.
\item \textbf{Transverse similarity is a partial correlation, and it has an exact reference scale.} The normalized transverse statistic $Z_t$ is the cross-sectional partial correlation of two signals controlling for realized returns, with a classical $t_{d-3}$ null. Persistent departures point to a shared feature family or pipeline, which is worth knowing even when the pair's PnLs are uncorrelated.
\item \textbf{Transverse structure is not inert once you combine signals.} For a single signal traded at fixed norm, the transverse component contributes nothing to that date's PnL. For a normalized or risk-budgeted combination, it enters through the denominator and determines how much IC or payoff-per-unit-risk the combination retains (\cref{sec:6}).
\item \textbf{Keep the two metrics separate.} Ordinary IC and partial correlation live in the Euclidean metric; payoff per unit risk lives in the risk-model metric. Report both. They agree for every signal configuration only when the risk model is isotropic on the signal span.
\item \textbf{Multiple correlation is a span diagnostic, not a library score.} Raw ex post $\mathcal R_t^2$ rises mechanically with the number of signals under a no-information null; use the adjusted version, and only as a null-relative measure of Euclidean span coverage. It contains realized returns and is not a trading objective.
\item \textbf{Treat the pairwise table as a multiple-testing problem.} With $K$ signals there are $K(K-1)/2$ pairs, each estimated from serially dependent data. Use block or stationary bootstraps or HAC standard errors, and control the false-discovery rate before acting on a ranking.
\end{itemize}
\end{practicebox}

\Needspace*{8\baselineskip}
\section{Technical remarks}\label{sec:13}

\begin{enumerate}
\item \textbf{Measurable selection.} No globally continuous choice $Q\mapsto G(Q)$ is needed. On a finite sample, choose $G_t$ date by date. More generally, a measurable selection suffices, and every reported quantity is independent of the residual frame choice.
\item \textbf{Horizon dependence.} Both $Q_t$ and the transverse decomposition depend on the forward-return horizon.
\item \textbf{Ex post status.} Because $G_t$ and $Q_t$ depend on realized forward returns, this is an ex post diagnostic, not a tradable transformation.
\item \textbf{Implementation.} The rotation itself is never computed; every quantity follows from inner products, return magnitude, and ordinary partial-correlation formulas.
\end{enumerate}

\phantomsection
\section*{Acknowledgments and AI-assistance disclosure}
\addcontentsline{toc}{section}{Acknowledgments and AI-assistance disclosure}

The author developed and revised this paper through iterative dialogue with ChatGPT, an AI system developed by OpenAI. Claude, an AI system developed by Anthropic, was used to generate repeated critical reviews of successive drafts, including checks of stated algebra and independent execution of the synthetic experiment. These systems are not authors or independent peer reviewers and cannot take responsibility for the work. Marc Nunes directed the inquiry, selected and evaluated the arguments, and accepts full responsibility for the paper's content, accuracy, and conclusions.

\clearpage
\appendix
\section{Reproduction code for the synthetic illustration}\label{app:code}

The following NumPy code fixes both the generator and draw order used in \cref{sec:9}.

\begin{lstlisting}[language=Python]
import numpy as np
from scipy.stats import spearmanr

rng = np.random.default_rng(20260906)
d, T, K = 500, 1000, 12
m = d - 2

innovation_sd = 0.045  # replace by 0.012, 0.025, 0.070 for sensitivity rows
alpha = np.linspace(0.015, 0.025, K)
f = rng.normal(size=T)
ell = np.linspace(-0.7, 0.7, K)
eps = rng.normal(size=(T, K))
p = alpha + innovation_sd * (
    f[:, None] * ell[None, :]
    + eps * np.sqrt(1 - ell**2)[None, :]
)
a = np.exp(0.35 * np.abs(f) + 0.25 * rng.normal(size=T))
a /= a.mean()

family = np.arange(K) % 3
beta = np.array([0.15, 0.35, 0.60, 0.80] * 3)[:K]
U = np.empty((T, K, m))

for t in range(T):
    centers = rng.normal(size=(3, m))
    centers /= np.linalg.norm(centers, axis=1)[:, None]
    noise = rng.normal(size=(K, m))
    noise /= np.linalg.norm(noise, axis=1)[:, None]
    directions = (
        beta[:, None] * centers[family]
        + np.sqrt(1 - beta**2)[:, None] * noise
    )
    directions /= np.linalg.norm(directions, axis=1)[:, None]
    U[t] = np.sqrt(1 - p[t] ** 2)[:, None] * directions

r_squared = np.empty(T)
for t in range(T):
    gram = np.outer(p[t], p[t]) + U[t] @ U[t].T
    r_squared[t] = p[t] @ np.linalg.pinv(gram) @ p[t]

adjusted_r_squared = 1 - (1 - r_squared) * (d - 1) / (d - 1 - K)

rows = []
for i in range(K):
    for j in range(i + 1, K):
        ic_product = p[:, i] * p[:, j]
        transverse = np.einsum("tm,tm->t", U[:, i], U[:, j])
        signal_similarity = (ic_product + transverse).mean()
        pnl_correlation = np.corrcoef(a * p[:, i], a * p[:, j])[0, 1]
        regime_term = np.cov(a * a, ic_product, ddof=0)[0, 1]
        rows.append(
            (i + 1, j + 1, signal_similarity, ic_product.mean(),
             transverse.mean(), pnl_correlation, regime_term)
        )

rows = np.asarray(rows)
print(spearmanr(rows[:, 2], rows[:, 5]).statistic)
print(spearmanr(rows[:, 3], rows[:, 5]).statistic)
print(r_squared.mean(), adjusted_r_squared.mean(), K / (d - 1))
\end{lstlisting}

\phantomsection
\printbibliography[heading=bibintoc]

\end{document}